\documentclass[journal]{IEEEtran}

\usepackage{amsmath,amssymb,amsthm}
\usepackage{array}
\usepackage{algorithm}
\usepackage{algorithmic}
\usepackage{pseudocode}
\usepackage{multirow}
\usepackage{booktabs}

\usepackage{graphicx}
\usepackage{subfig}
\usepackage{amsmath,amssymb,amsthm}
\usepackage{epstopdf}
\usepackage{verbatim}
\usepackage{caption}

\begin{document}
\newtheorem{theorem}{Theorem}
\newtheorem{definition}[theorem]{Definition}
\newtheorem{remark}[theorem]{Remark}
\newtheorem{example}[theorem]{Example}
\newtheorem{lemma}[theorem]{Lemma}
\newtheorem{proposition}[theorem]{Proposition}
\newtheorem{corollary}[theorem]{Corollary}

\title{Deterministic compressed sensing matrices: Construction via Euler Squares and applications}
\author{R. Ramu Naidu, Phanindra Jampana and C. S. Sastry
\thanks{R. Ramu Naidu and C. S. Sastry are with the Department
of Mathematics, Indian Institute of Technology, Hyderabad, Telangana,
502285, India, e-mail: ma11p003@iith.ac.in, csastry@iith.ac.in}
\thanks{Phanindra Jampana is with the Department
of Chemical Engineering, Indian Institute of Technology, Hyderabad, Telangana,
502285, India, e-mail: pjampana@iith.ac.in}
\thanks{Copyright (c) 2015 IEEE. Personal use of this material is permitted.
However, permission to use this material for any other purposes must be
obtained from the IEEE by sending a request to pubs-permissions@ieee.org.}}

\maketitle


\begin{abstract}
In Compressed Sensing the matrices that satisfy the Restricted Isometry Property (RIP) play an important role.
 But  to date, very few results  for designing such matrices are available. For applications such as multiplier-less data compression, binary sensing matrices are of interest. The present work constructs deterministic and binary sensing matrices using Euler Squares. In particular, given a positive integer $m$ different from $p, p^2$ for a prime $p$, we show that it is possible to construct a binary sensing matrix  of size $m \times c (m\mu)^2$, where $\mu$ is the coherence parameter of the matrix and $c \in [1,2)$. The matrices that we construct have small density (that is, percentage of nonzero entries in the matrix is small) with no function evaluation in their construction, which support algorithms with low computational complexity. Through experimental work, we show that our binary sensing matrices can be used for such  applications as content based image retrieval. Our simulation results demonstrate that the Euler Square based CS matrices give better performance than their Gaussian counterparts. 

\end{abstract}

\begin{IEEEkeywords}
Compressed Sensing, Coherence, RIP, Binary sensing matrices, Euler Squares, CBIR.
\end{IEEEkeywords}

\section{Introduction}
Compressed Sensing (CS) aims at recovering high dimensional sparse vectors from considerably fewer linear measurements.
The problem of sparse recovery through $l_{0}$ norm minimization  is not  tractable. E. Candes \cite{can_2008}, D. Donoho \cite{Donoho_2006},
 have made pioneering contributions reposing the problem as a simple Linear Programming Problem (LPP). They have
then established the conditions that ensure the stated equivalence between the original $l_{0}$ problem and its reposed
version. It is known that Restricted Isometry Property (RIP) is one sufficient condition to ensure the equivalence. Random matrices with 
Gaussian or Bernoulli entries have been shown to satisfy RIP with high probability \cite{Richard_2008}.

\par In the recent literature on CS \cite{apple_2009, bryant_2014,  Cal_2010}, the deterministic construction of CS matrices has gained momentum. A. Amini et. al. \cite{amini_2011}, R. Devore \cite{Ronald_2007}, P. Indyk \cite{ind_2008} and  S. Li et. al. \cite{li_2012} have constructed deterministic binary sensing matrices using ideas from algebra and graph theory. In the present work, however, we attempt to construct deterministic binary sensing matrices using Euler Squares. The advantages of the proposed methodology are as follows:
\begin{itemize}
 \item Matrices of general row size (different from a prime and its square) and small density are constructed.
\item Simplicity in construction is achieved 
\end{itemize}
Here, density is defined as the ratio of number of nonzero entries to the total number of entries of the matrix. Sparse sensing matrices may contribute to fast processing with low computational complexity in compressed sensing \cite{gil_2010}.
\par In the recent past, Gaussian and $\pm 1$ Bernoulli random matrices (which are shown to satisfy RIP \cite{Richard_2008}) have been used to project data into lower dimension for the purpose of classification \cite{monica_2011}. There are, however, the following advantages of using deterministic binary matrices for the stated purpose:
\begin{itemize}
 
 \item Binary matrices being sparse and possessing $0, 1$ as elements provide multiplier-less and faster dimensionality reduction operation, which is not possible with Gaussian matrices
\item There is a nonzero probability of non-compliance with RIP for Gaussian 
  matrices
\end{itemize} 
\par The problem (\cite{CS_2013, Murala_2012}) of searching for similar images in a large image repository based on content is called Content Based Image Retrieval (CBIR). CBIR has several multimedia applications like text based retrieval (such as google search) and retrieval from medical databases. As images of large size typically involve more complexity, one needs to project them to lower dimensional spaces. It is demonstrated that the proposed binary sensing matrices project data into lower dimensional spaces in such a way that the reduced vectors are useful for the purpose of CBIR. Moreover, the proposed dimensionality reduction technique through binary sensing matrices allows for reconstruction, which is important for tele-medicine (\cite{ag_2010, sandy_2013}). The other dimensionality reduction techniques and even the sparsity seeking Dictionary based methods \cite{CS_2013} in general do not provide this advantage.

\par The paper is organized into several sections. In section \ref{sec:2},
basic CS theory and the equivalence between $l_{0}$~norm and $l_1$~norm 
problems are given. In section \ref{sec:3}, the deterministic construction 
procedure of binary sensing matrices using Euler Squares is presented.
Section \ref{sec:4} gives comparison with existing constructions. 
In section \ref{sec:5}, an application to content based image retrieval
is demonstrated.  Concluding remarks are given in section \ref{sec:7}.  

\section{Sparse recovery from linear measurements} \label{sec:2}
As stated already, CS refers to the problem of reconstruction of an unknown vector $u \in \mathbb{R}^{M}$ from
linear measurements $y = (\langle{u,\phi_{1}\rangle},\ldots,\langle{u,\phi_{m}\rangle})
\in \mathbb{R}^{m}$ with $\langle{u,\phi_{j}}\rangle$ being the inner-product between $u$ and $\phi_{j}$. The basic objective
in CS is to design a recovery procedure based on the sparsity assumption on $u$ when the number of measurements $m$ is
very small compared to $M$. Sparse representations have merit for various applications \cite{ CS_2013, Elad_2010, Sek_2014, Chand_2014, hari_2014, hari_2013} in areas such as
image/signal processing and numerical computation. \par A vector $u \in \mathbb{R}^{M}$ is $k-$sparse if it has at most $k$
nonzero coordinates. Let $\Vert{v} \Vert_0$ stand for $\left|\left\{i \mid v_{i}\neq 0\right\} \right|$. The problem of obtaining the sparse vector from its linear measurements may be posed as
\begin{displaymath}
\label{eqn:l_0}
P_0:\min_{v} \Vert{v} \Vert_0 \; \mbox{subject to} \quad \Phi v=y.
\end{displaymath}
This $l_0-$minimization problem is NP-hard \cite{Donoho_2006} in general. 
There are greedy algorithms for solving $P_0$ problem and orthogonal matching 
pursuit (OMP) is one of the popular methods \cite{Elad_2010}. 
Several researchers (for example, D. Donoho \cite{Donoho_2006} and E. Candes \cite{can_2008}) have established the conditions that ensure the recovery of solution to $P_0$ from :
\begin{displaymath}
P_1:\min_{v} \Vert{v} \Vert_1 \; \mbox{subject to} \quad \Phi v=y.
\end{displaymath}
Here $\Vert{v} \Vert_1$ denotes the $l_{1}$-norm of the vector $v \in \mathbb{R}^{M}$. Denote the solution to $P_{1}$ by $f_{\Phi}(y)$ and solution to $P_{0}$ by $u^{0}_{\Phi}(y) \in \mathbb{R}^{M}$. 
\subsection{On the equivalence between $P_{0}$ and $P_{1}$ problems} 
The coherence $\mu(\Phi)$ of a given matrix $\Phi$ is the largest absolute inner-product between different normalized  columns of $\Phi$. Denoting the $k$-th column in $\Phi$ by $\phi_k$, one defines $\mu(\Phi)$, the coherence, as
$\displaystyle \mu(\Phi)= \max_{1\leq\; i,j \leq\; M,\; i\neq j} \frac{| 
\langle \phi_i, \phi_j\rangle |}{\Vert \phi_i\Vert_{2} \Vert \phi_j \Vert_2}.$
\noindent For a matrix of size $m \times M$, $\mu$ satisfies: $\mu \geq \sqrt{\frac{M-m}{m(M-1)}}$, called the Welch bound. 
It is known \cite{Elad_2010} that for $\mu$-coherent matrices $\Phi$, one has
$\displaystyle u^{0}_{\Phi}(y) = f_{\Phi}(y) = u$,
provided $u$ is $k-$sparse with $k < \frac{1}{2}\; (1+\frac{1}{\mu}).$

\par The Restricted Isometry Property (RIP) plays an important role \cite{can_2008} in CS as it establishes the equivalence between the $P_0$ and $P_1$ problems. An $m \times M$ matrix $\Phi$ (which we refer to a CS matrix) is said to satisfy the Restricted Isometry Property(RIP) of order $k$ with constant $\delta_{k}$ if for all $k-$sparse vectors $x\in \mathbb{R}^{M}$, we have
\begin{equation}
\label{eqn:ri}
(1-\delta_{k}) \left\|x\right\|^{2}_{l_{2}} \leq \left\|\Phi x\right\|^{2}_{l_{2}} \leq (1+\delta_{k}) \left\|x\right\|^{2}_{l_{2}}.
\end{equation}
\noindent It is known \cite{can_2008}\cite{Donoho_2006} that the RIP along with some conditions on $\delta_k$ imply the equivalence between the $P_0$ and $P_1$ problems. 
\par One of the important problems in CS theory deals with constructing CS matrices that satisfy RIP for the largest possible range of $k$. It is known that the random constructions satisfy the RIP for the largest possible range on sparsity, which is $k = O(\frac{m}{\log(\frac{M}{k})})$ \cite{Richard_2008}.
\subsection{Existing deterministic constructions}\label{subsec:1}
To the best of our knowledge, designing good deterministic constructions of 
RIP matrices is still a very interesting problem. R. Devore \cite{Ronald_2007} has constructed deterministic binary sensing  matrices of size $p^2 \times p^{r+1}$ with coherence $\frac{r}{p}$, where $p$ is a prime power and $0 < r <p$. 
Later on, S. Li, F. Gao et. al. \cite{li_2012} have generalized the work in \cite{Ronald_2007}, constructing binary sensing matrices of size $|\mathcal{P}|q \times q^{\mathcal{L}(G)}$, where $q$ is any prime power and  $\mathcal{P}$ is the set of all rational points on algebraic curve $\mathcal{X}$ over finite field $\mathbb{F}_q$. P. Indyk \cite{ind_2008} has constructed binary sensing matrices using hash functions and extractor graphs with sizes $r2^{O(\log\log n)^{O(1)}} \times n$, where $r \ll n$. A. Amini et. al. \cite{amini_2011} have constructed binary sensing matrices using OOC codes. In all these constructions, CS matrices are given for specific sets of row sizes. 
J. Bourgain et. al. \cite{bourgain_2011} have constructed RIP matrices of order $k \geq m^{\frac{1}{2} + \epsilon}$, for some $\epsilon > 0$ and $M^{1 - \epsilon} \leq m \leq M$ using additive combinatorics. It is remarkable that this construction overcomes the natural barrier $k = O(m^{\frac{1}{2}})$ for those based on coherence. J. L. Nelson et. al \cite{Lidl_2011} have given lower and upper bounds on maximal possible column size in terms of fixed row size $m$ and coherence $\mu$.  They have given a class of deterministic CS matrices of size $p \times p^{r}$  with coherence $\frac{r-1}{\sqrt{p}}$, where $p$ is a prime number. R. Calderbank et.al. \cite{cal_2010, Cal_2010} have constructed CS matrices of size $2^{l} \times 2^{(r+2)l}$  with coherence $2^{r- \frac{l}{2}}$, using the Delsarte-Goethals codes, where $l$ is an odd number and $0 \leq r \leq \frac{l-1}{2}$ is a constant integer. M. A. Herman et.al. \cite{herman_2009} have constructed Gabor frame of size $p \times p^{2}$ with coherence $\frac{1}{\sqrt{p}}$, using Alltop sequence, where $p\geq 5$ is any prime number. The constructions, possessing sparsity $k=\sqrt{m}$ have been given by Calderbank et. al. in \cite{apple_2009}  based on tight frames such as $m \times m^{2}$ chirp matrices and Alltop Gabor frames \cite{baj_2010}, here $m$ is prime or prime power. The authors in \cite{Li_2014, Nam_2013} have constructed matrices which achieve the square root bottleneck asymptotically.

 In the present work, binary sensing matrices are constructed using 
Euler Squares. In particular, given any positive integer $m\neq p, p^2$ (for prime $p$), the procedure allows to construct binary CS matrix of size $m \times M$, where $M$ is $c(m\mu)^2$, with $\mu$ being the coherence parameter of the matrix (which is dependent on $m$ as explained in sub-sequent sections) and $c \in [1,2)$. As the proposed methodology does not involve any function evaluations and gives matrices of small density, this method involves less complexity as compared to some of the existing ideas stated above. The 
applicability of construction is demonstrated through an application to CBIR.
\section{Euler Squares for constructing CS matrices}\label{sec:3}

\subsection{Euler Squares}
An Euler Square of order $n$, degree $k$ and index $n,k$ is a square array of $n^2$ $k-$ads of numbers, $(a_{ij1},a_{ij2},\ldots,
a_{ijk})$, where $a_{ijr} \in \{0,1,2,\ldots,n-1\}; r = 1,2,\ldots,k; i,j = 1,2,\ldots,n; n > k; a_{ipr} \neq a_{iqr}$ and $a_{pjr} \neq a_{qjr}$ for $p \neq q$ and $(a_{ijr}+1)(a_{ijs}+1) \neq (a_{pqr}+1)(a_{pqs}+1)$ for $i \neq p$ and $j \neq q.$
\par Harris F. MacNeish \cite{euler_1922} has constructed Euler Squares 
for the following cases:\\
1. Index $p,p-1$, where $p$ is a prime number.\\
2. Index $p^r, p^{r}-1$, for $p$ prime.\\
3. Index $n,k$, where $n = 2^{r}p^{r_{1}}_{1}p^{r_{2}}_{2}\ldots, p^{r_{l}}_{l}$ for distinct odd primes $p_{1},p_{2},\ldots, p_l$. Here, $k + 1$ equals the least of the numbers $2^{r},p^{r_{1}}_{1},p^{r_{2}}_{2},\ldots ,p^{r_{l}}_{l}.$\\
For example, the Euler Square of index $3,2$ is as follows:

$$0,0 \ \ \ \  1,1 \ \ \ \ 2,2 $$
$$1,2 \ \ \ \  2,0 \ \ \ \ 0,1 $$
$$2,1 \ \ \ \  0,2 \ \ \ \ 1,0 $$
\begin{lemma}\cite{euler_1922}\label{lem:e}
Let $k' < k$. Then the existence of the Euler Square of index $n,k$ implies that the Euler Square of index $n,k'$ exists.
\end{lemma}
\subsection{Deterministic constructions via Euler Squares}\label{subsec:eu}
\noindent Using the concept of Euler Square of index $n,k$, deterministic 
binary CS matrices $\Phi$ that possess small coherence can be constructed. 
In particular, for $n \geq 3, k \geq 2,$ a CS matrix of size $nk \times n^2$, which may be treated as a block matrix consisting of $k$ number of $n \times n^2$ blocks can be obtained. Each column in $\Phi$ corresponds to a $k$-ad in the Euler Square. The matrix $\Phi$ is defined as follows:
For $1 \leq i \leq nk, 1 \leq j \leq n^2$,
\begin{equation}
\label{eqn:Euler Main}
\phi_{ij}= \left\{
\begin{array}{cc}
 1 & \mbox{if} \;(a_j)_{\left\lfloor\frac{i-1}{n}\right\rfloor + 1} \equiv i-1 (\mbox{mod} \;n) \\
0 & \mbox{otherwise},
\end{array}
\right \},
\end{equation}
where $(a_{j})$ is the $j^{th}$ $k-$ad, $(a_{j})_{l}$ is $l^{th}$ element in $j^{th}$ $k-$ad and $\left\lfloor x\right\rfloor$ is the largest integer not greater than  $x$. To summarize, for $j=1,2, \dots, n^{2},$ the $j^{th}$ column $\phi_{j}$ of $\Phi$ is generated as an $nk-$ binary vector from the $j^{th}$ $k$-ad $(a_{j})$ with 1 occurring at the positions $(l-1)n+((a_{j})_{l}+1)$ for $l= 1,2,\ldots,k$. 
\par There are exactly $k$ ones in each column of $\Phi$ and size of the matrix is $nk\times n^2$. For example, the $6 \times 9$ matrix given by the Euler Square of index $3,2$ is as follows:
\begin{center}
\[ \left( \begin{array}{ccccccccc}
1&0&0&0&0&1&0&1&0 \\
0&1&0&1&0&0&0&0&1 \\
0&0&1&0&1&0&1&0&0 \\
1&0&0&0&1&0&0&0&1 \\
0&1&0&0&0&1&1&0&0 \\
0&0&1&1&0&0&0&1&0 \\
\end{array} \right) \]
\end{center}

\noindent The following lemma finds the bound on the coherence of $\Phi$. \\
\begin{lemma}
\label{thm:lem}
The coherence of $\Phi$ is at most equal to $\frac{1}{k}$.
\end{lemma}
\begin{proof}
Suppose there exist two columns $\phi_c, \phi_d$ of $\Phi$ such that the overlap between them is at least 2. That is, if $u_c, u_d$ are the support vectors of $\phi_c, \phi_d$, then $\left|u_c \cap u_d\right| \geq 2$, which implies that there exist two $k-$ads $(a_c) = (a_{ij1},a_{ij2},\ldots,a_{ijk}), (a_d) = (a_{pl1},a_{pl2},\ldots,a_{plk})$ with $i = \left\lfloor\frac{c-1}{n}\right\rfloor +1, j = c - \left\lfloor\frac{c-1}{n}\right\rfloor n$ and $p = \left\lfloor\frac{d-1}{n}\right\rfloor +1, l = d - \left\lfloor\frac{d-1}{n}\right\rfloor n$ such that 
\begin{equation}\label{eqn:ci}
a_{ijr} = a_{plr} \text{~and~} a_{ijs} = a_{pls} \text{~for some~}1 \leq r,s \leq k.
\end{equation}
\textbf{Case 1:} Suppose $i = p$ and $j \neq l$, from (\ref{eqn:ci}),  $a_{pjr} = a_{plr}$ and $a_{ijs} = a_{ils}$ for $j \neq l$, which is a contradiction from the definition of Euler Square.\\
\textbf{Case 2:} Suppose $i \neq p$ and $j = l$. From (\ref{eqn:ci}), $a_{ilr} = a_{plr}$ and $a_{ijs} = a_{pjs}$ for $i \neq p$, which is again a contradiction.\\
\textbf{Case 3:} 
Suppose $i \neq p$ and $j \neq l$, from (3), $a_{ijr} = a_{plr}$ and $a_{ijs} = a_{pls}$. Since, $a_{ijr}, a_{pls} \in \{0,1,\ldots,n-1\}$ for all $1 \leq r,s \leq k$, we have $a_{ijr}+1 = a_{plr}+1$ and $a_{ijs}+1 = a_{pls}+1$, which implies that $(a_{ijr}+1)(a_{ijs}+1) = (a_{plr}+1)(a_{pls}+1)$ for $i \neq p$ and $j \neq l$, which is a contradiction from the definition of Euler Square.
\par Hence there are no two such $k-$ads satisfying (\ref{eqn:ci}). Therefore, the overlap between any two columns of $\Phi$ is at most 1. Since each column in $\Phi$ contains fixed number of $k$ ones, it follows that the coherence of $\Phi$, $\mu(\Phi)$, is at most equal to $\frac{1}{k}$.
\end{proof}

\noindent {\bf Remark 1:} The maximum possible column size of any binary matrix is $\frac{{m \choose r}}{{k \choose r}}$ \cite{amini_2011}, where $m$ is the row size, $k$ is the number of ones in each column and $r-1$ is the overlap between any two columns. Euler Square of index $n,k$ results in a 
binary sensing matrix of size $nk \times n^2$. In this matrix each column contains $k$ number of ones and coherence is at most $\frac{1}{k}$. The maximum possible column size is thus 
\begin{equation}\label{maxcol}
\frac{{nk \choose 2}}{{k \choose 2}} = \Theta(n^2)\footnote{$a= \Theta(b)$ implies that, there exist two constants $c_{1}, c_{2}$ such that $c_{1}b \leq a \leq c_{2}b$.}= \Theta((m\mu)^2).
\end{equation}
 For the Euler Square based matrix $m= nk, M= n^2$ and $\mu = \frac{1}{k}$, and hence $M= (m\mu)^2$, which is in the order of maximum possible column size. Hence the aspect ratio is also in the maximum possible order. 
The density for these matrices is $\frac{1}{n}$.

\noindent The following proposition \cite{bourgain_2011} relates the RIP constant $\delta_{k'}$ and $\mu.$
\begin{proposition}
\label{thm:pro}
Suppose that $\phi_{1},\ldots,\phi_{m}$ are the unit norm columns of the matrix $\Phi$ possessing the coherence $\mu$. Then $\Phi$ satisfies the RIP of order $k'$ with constant $\delta_{k'} = (k'-1)\mu$. 
\end{proposition}
\noindent From lemma \ref{thm:lem} and Proposition \ref{thm:pro}, it follows that the matrix $\Phi$ so constructed satisfies RIP.
\begin{theorem}
\label{thm:main}
The matrix $\Phi_{0} = \frac{1}{\sqrt{k}}\Phi$ satisfies RIP with $\delta_{k'} = \frac{k'-1}{k}$ for any $k' < k + 1$. 
\end{theorem}
\noindent {\bf Remark 2:}
$m = nk, M = n^2$ and $k' < k + 1$ gives $k = \frac{m}{n}, n = \sqrt{M}$ which implies that $k = \frac{m}{\sqrt{M}}$. Consequently $k'(m, M) < \frac{m}{\sqrt{M}} + 1.$ 
 
\par The previous procedure can be generalized for any row size $m$ different 
from a prime and its square. The following theorem summarizes the main result. 
\begin{theorem}
\label{thm:main1}
Suppose $m$ is any positive integer different from $p, p^2$ for a prime $p$. Then there exists a binary sensing matrix of size $m \times M$ with coherence $\mu = \frac{\sqrt{M}}{m}$.
\end{theorem}
\begin{proof}
\textbf{Case-1:} If $m = p^{i}, i>2$ then $m$ can be written as $m = p^{i-1}p$. Since $i > 2$, we have $(p^{i-1}-1) > p$. From  \cite{euler_1922}, it is known that the Euler Square of index $p^{i-1}, p^{i-1}-1$ exists. Since $(p^{i-1}-1) > p$, from lemma~\ref{lem:e} Euler Square of index $p^{i-1}, p$ exists. If we construct the binary matrix using this Euler Square, then the row size of it becomes $p^{i}$. \\
\textbf{Case-2:} Let $m$ be any integer such that $m \neq p^{i}$ for $i>2$. From the fundamental theorem of arithmetic $m$ is factorized as $m = 2^{r}p^{r_{1}}_{1}p^{r_{2}}_{2}\ldots p^{r_{l}}_{l}$. Let $k'= \min\{2^{r},p^{r_{1}}_{1},p^{r_{2}}_{2},\ldots ,p^{r_{l}}_{l}\}$ and $k = k'-1$. With out loss of generality assume that $k' = 2^r$ (The following arguments hold true even if  $k' = p^{r_{s}}_{s}$ for some $s \in \{1,2,\ldots, l\}$). From the construction of MacNeish \cite{euler_1922}, it is known that the Euler Square of index $m, k$ exists. Let $m_{1} = p^{r_{1}}_{1}p^{r_{2}}_{2}\ldots p^{r_{l}}_{l}$ and without loss of generality assume that $p^{r_1}_{1}=\min\{p^{r_{1}}_{1},p^{r_{2}}_{2},\ldots ,p^{r_{l}}_{l}\}$. Since $2^{r} \leq (p^{r_1}_{1}-1)$, Euler Square of index $m_{1}, (p^{r_1}_{1}-1)$ exists, which in view of lemma~\ref{lem:e} implies that the  Euler Square of index $m_{1}, 2^{r}$ exists. The binary CS matrix constructed through this Euler Square has a row size of $m_{1}2^{r}$, which is 
$m$.
\end{proof}
\noindent {\bf Remark 3:} In both the cases above $M=(m\mu)^2$, which is in the order of maximum possible column size as stated in Remark-1   and they allow for successful signal recovery of sparsity $k'(m, M) < \frac{m}{\sqrt{M}} + 1$.\\
\par The afore-described construction methodology results in a binary CS matrix of size $nk \times n^2$ from an Euler Square of index $n,k$.
 As explained in Remark 1, the maximum possible column size for our construction in terms of row size $nk$ and coherence $\frac{1}{k}$ is $\frac{{nk \choose 2}}{{k \choose 2}} = \Theta(n^2)$. Since, $n^{2} \leq \frac{{nk \choose 2}}{{k \choose 2}} < 2n^{2}$, we have $c_{1} = 1$ and $c_{2} = 2$.  
In Theorem 5, we give the construction for $c =1$ and in the next section, when $n$ is of the form $2^{r}p^{r_{1}}_{1}p^{r_{2}}_{2}\ldots p^{r_{l}}_{l}$ (that is, $n$ is not a power of single prime) we try to increase the column size of above construction to $c(n^{2})$, where $c \in (1, 2)$, leaving the row size and the coherence intact.
 The procedure discussed above uses smallest factor, while in increasing the column size, other factors are also used.
\subsection{Column extension of CS matrices constructed from Euler Squares}\label{subsec:col}
\par As before, let $n = 2^{r}p^{r_{1}}_{1}p^{r_{2}}_{2}\ldots, p^{r_{l}}_{l}$, $k'$= min$\{2^{r},p^{r_{1}}_{1},p^{r_{2}}_{2},\ldots ,p^{r_{l}}_{l}\}$ and $k = k'-1$.  Suppose $\Phi^{(0)}$ is the matrix of size $nk \times n^2$ that we obtain from the Euler Square of index $n,k$. Without loss  of generality, assume that $2^{r} < p^{r_{1}}_{1} < p^{r_{2}}_{2} < \ldots < p^{r_{l}}_{l}$, let $k_1=$ max$\{2^{r},p^{r_{1}}_{1},p^{r_{2}}_{2},\ldots ,p^{r_{l}}_{l}\}= p^{r_{l}}_{l}$. Clearly, $n_{1}=\frac{n}{k_1}=2^{r}p^{r_{1}}_{1}p^{r_{2}}_{2}\ldots p^{r_{l-1}}_{l-1}$. Since the Euler Square of index $n_{1},k$ exists and using it we construct $n_{1}k \times n_1^2$ matrix, say $\Phi^{(1)}$, with coherence at most $\frac{1}{k}$. Now using these two matrices $\Phi^{(0)}$ and $\Phi^{(1)}$, a new matrix $\Psi^{(1)}$ is obtained as detailed below:
\par The matrix $\Phi^{(0)}$ may be viewed as a block matrix possessing $k$ blocks with each block being of size $n \times n^2$. Since $k_{1} > k$ and $n_{1}k=\frac{n}{k_1}k$, we have $n_{1}k<n.$  The suitably zero-padded $\Phi^{(1)}$ may be added as an additional column block to $\Phi^{(0)}$ to generate a matrix of size $nk \times (n^2+n_1^2)$. 
One may obtain $k$ number of zero-padded column blocks by placing $\Phi^{(1)}$ in the corresponding row-block-locations of $\Phi^{(0)}$. Using all these column blocks, we obtain a new binary matrix of size $nk \times (k n_1^2)$, say $\Psi^{(1)}$.
From Lemma \ref{thm:lem}, the inner-product between any two columns in $\Phi^{(0)}$ and $\Psi^{(1)}$ is at most 1. Further the inner-product between any two columns one from $\Phi^{(0)}$ and other from $\Psi^{(1)}$ is at most 1, this is because the $k$ nonzero entries of a column from $\Psi^{(1)}$ fall only in one  of the blocks of $\Phi^{(0)}$ which has only one nonzero entry.
 \par The elements of the matrix $\Psi^{(1)}$ may be generated as follows: 
for $t=1,2, \ldots, k; \; (t-1)n^{2}_{1}+1 \leq j \leq tn^{2}_{1}, \; 1\leq i\leq (t-1)n, \; \psi^{(1)}_{ij}= 0$, also for $(t-1)n^{2}_{1} + 1 \leq j \leq tn^{2}_{1}, (t-1)n +1 \leq i \leq (t-1)n + n_{1}k$,
\begin{equation}
\psi^{(1)}_{ij}= \begin{cases}
 1 & \mbox{if} \;(a_j)_{\left\lfloor\frac{i-(1+(t-1)n)}{n_{1}}\right\rfloor + 1} 
 \equiv i-(1+(t-1)n)\\
 & ~~~~~~~~~~~~~~~~~~~~~~~~~~~~~~(\mbox{mod} \;n_{1}) \\
0 & \mbox{otherwise},
\end{cases}\end{equation}
and for $(t-1)n^{2}_{1} + 1 \leq j \leq tn^{2}_{1}, (t-1)n + n_{1}k +1 \leq i \leq nk, \;\; \psi^{(1)}_{ij}= 0$,
where $(a_{j})$ is the $j^{th}$ $k-$ad in the Euler Square of index $n_{1},k$.
\par Similarly, using $k_2=$ max$\{2^{r},p^{r_{1}}_{1},p^{r_{2}}_{2},\ldots ,p^{r_{l-1}}_{l-1}\}= p^{r_{l-1}}_{l-1}$, $n_2 = \frac{n_1}{k_2}$, a CS matrix of size $n_2 k \times n_2^2$  (say $\Phi^{(2)}$) is obtained.  The matrix $\Phi^{(1)}$ may be viewed as a block matrix possessing $k$ blocks with each block being of size $n_{1} \times n_{1}^2$. As shown in Fig. 1, the matrix $\Psi^{(1)}$ contains $k$ such $\Phi^{(1)}$ matrices, therefore there are $k^{2}$ blocks each of size $n_{1} \times n_{1}^2$ in $\Psi^{(1)}$. Since $k_{1} > k_{2} > k$ and $n_{2}k=\frac{n_{1}}{k_2}k$, we have $n_{2}k<n_{1}.$  The suitably zero-padded $\Phi^{(2)}$ may be added in these $k^{2}$ blocks as additional column blocks to $\Psi^{(1)}$ as shown in Fig. 1 and generate a new matrix, say $\Psi^{(2)}$, of size $nk \times k^2 n^{2}_2$. Continuing this way till Euler Square of index $n_l, k$, a matrix, 
say $\Psi^{(l)}$, of size $nk \times k^l n_{l}^2$ is obtained at the $l^{th}$ 
stage. Finally, concatenating $\Phi^{(0)}, \Psi^{(1)}, \ldots, \Psi^{(l)}$, 
results in a matrix with coherence at most $\frac{1}{k}$ and column size 
$(n^{2} + n^{2}_{1}k + n^{2}_{2}k^{2} + n^{2}_{3}k^{3}+ \ldots +n^{2}_{l}k^{l}) \geq n^{2}  \left[\frac{1-(\frac{k}{k^{2}_{1}})^{r+1}}{1-(\frac{k}{k^{2}_1})}\right]$. This methodology is explained diagrammatically in Figure \ref{fig:block} for $l=2$, that is, $n = 2^{r} p_1^{r_1} p_2^{r_2}$.  

\begin{figure}[!ht]
\begin{center}
\includegraphics[width=9cm, height=5cm]{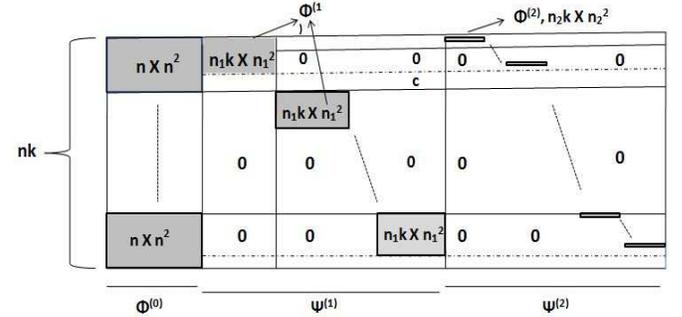}
\caption{Block diagram depicting the extension of the column size of CS matrix constructed from the Euler Square. The way the blocks, obtained from several Euler Squares, are arranged ensures that the column size gets enlarged without affecting the coherence of initial matrix.} 
\label{fig:block}
\end{center}
\end{figure}   

The column extended matrix has RIP compliance as summarized by the following theorem.
\begin{theorem}
\label{thm:ma}
The matrix $[ \Phi^{(0)} \; \Psi^{(1)} \; \ldots \; \Psi^{(l)} ]$ satisfies the RIP with $\delta_{k'} = \frac{k'-1}{k}$ for any $k' < k + 1$. 
\end{theorem}
\noindent {\bf Remark 4:} Combining the results in Sections \ref{subsec:eu} and \ref{subsec:col}, it can be concluded that the constructed matrix attains 
column size $m \times c(m\mu)^2$ for some $c \in [1, 2).$ In particular, when $m$ is power of a single prime (that is, $m=p^l, l>2$), $c=1$, and in other cases $c$ can be made to lie in $(1,2)$, as justified in Section \ref{subsec:col}.\\

\section{Comparison with existing constructions}\label{sec:4}
\noindent As discussed already in the Section \ref{subsec:1}, existing 
methods provide binary CS matrices for specific set of integers. 
The present work, however, provides a procedure for constructing binary matrices for a large class of row sizes. From  Remark 1, the maximum possible column size is between $n^2$ and $2n^2$. Since $k \geq 2$, it follows that $(nk)^{2} \geq 4n^2$ and hence $M < m^2$. In the constructions provided by \cite{ amini_2011, bourgain_2011,  Nam_2013, Yu_2013}, $M < m^2$ holds, albeit for specific set of values of $m$. While in the construction of \cite{Ronald_2007}, $M > m^2$ holds. It is known \cite{mix_2011} that the Welch bound is not sharp in this case. Consequently the gain obtained in terms of increased column size does affect the coherence, which inturn restricts the sparsity of the solution to be recovered.  
 In addition to providing general row size, the present construction is simple in the sense that it does not involve function evaluations like in \cite{Ronald_2007} and gives matrices with small density, which support algorithms with low computational complexity. For example, to generate an Euler square matrix of size $p,p-1$, it is only required to store two cyclic permutations of length $p$ and $p-1$ respectively and for index $p^{i},p^{i-1}$,  it is sufficient to store at most $\frac{p^{i}}{2}$ permutations \cite{euler_1922}.
\par To the best of our knowledge, the constructions possessing sparsity $k '= \sqrt{m}$ (that is, coherence $\mu = \frac{1}{\sqrt{m}}$) exist for non-binary matrices with  row size  $m$ being prime or prime power \cite{apple_2009, baj_2010, cal_2010, Cal_2010}. While in our construction methodology, if we use the Euler Squares of index $p, p-j$ for $j=1,2$, then the size of the matrix $\Phi$ is $p(p-j) \times p^{2}$ and coherence $\mu$ is $\frac{1}{p-j}$. Since the row size $m$ of $\Phi$ is $ p(p-j)$, we have  $\mu = \frac{1}{\left\lfloor\sqrt{m}\right\rfloor}$. Hence, this matrix provides guarantees for signals of sparsity up to $ k' = \left\lfloor\sqrt{m}\right\rfloor$. This is true for any Euler Square of index $p^{i}, p^{i}-j$, where $p$ is a prime number, $i \geq 1,$ and $j=1,2.$  Hence we construct the binary matrices that provide guarantees for signals of sparsity up to $k' = \left\lfloor\sqrt{m}\right\rfloor$ for different class of row sizes such as $p^{i}(p^{i}-j)$, where $p$ is prime $i \geq 1,$ and $j=1,2.$
\par For an arbitrary binary matrix, if the inner-product between any two columns is at most 1, every column contains fixed ($\sqrt{m}$) number of ones (that is coherence is at most $\frac{1}{\sqrt{m}}$) and row size is $m$, then the maximum possible column size $M= O(m)$ as mentioned earlier. 
 
\par Using Euler squares, we have constructed  matrices that provide guarantees for signals of sparsity up to $ k' = \left\lfloor\sqrt{m}\right\rfloor$ with $M = O(m)$. 
Through an extension of our construction methodology, it is possible to  generate ternary matrices,  for which $k' = \sqrt{m}$ holds (that is, coherence is in the order of $\sqrt{\frac{M-m}{m(M-1)}} $ )  with column size $M = O( m^{\frac{3}{2}})$, which is explained in the following:
\par Suppose $\Phi$ is a matrix constructed from the Euler square of index $p^{i}, p^{i} - j$ for $i \geq 1, j=$ 1 or 2 and $H$ a Hadamard matrix of size $(p^{i} - j)$ or $(p^{i} - j) + 1$.  To get a ternary CS matrix (i.e. matrix containing $0,\pm 1$ as elements) $\Psi$,  for each column of $\Phi$, we replace each of
its $1$-valued entries with a distinct row of $H$ and $0$-valued entries with a zero row of same size. The  size of the matrix $\Psi$ becomes $p^{i}(p^{i} - j) \times p^{2i}(p^{i} - j)$. In this construction, it is very easy to check that the coherence of the matrix $\Psi$ is $\frac{1}{(p^{i} - j)}$, which is in the order of $\sqrt{\frac{M-m}{m(M-1)}} $ (Welch bound) with column size $M = O( m^{\frac{3}{2}})$.

With a view towards comparing the numerical performance against the standard 
Gaussian (with entries drawn from $\mathcal{N}(0, \frac{1}{m}))$ and Bernoulli (with entries $\phi_{ij} = \pm \frac{1}{\sqrt{m}}$, each with probability $\frac{1}{2}$) random matrices, binary matrices of size $55 \times 121$ and 
$230 \times 529$ are generated using the Euler Squares of indices 
$(11,5)$ and $(23,10)$. The OMP algorithm is used 
to solve the $l_0$ minimization problem for a signal $x$. Let $\tilde x$ denote the recovered solution. From the reconstruction, the Signal-to-noise ratio 
(SNR) of $x$ is computed using
\begin{displaymath}
SNR(x) = 10 . \log_{10} \biggl( \frac{\|x \|_2}{\|x - \tilde x \|_2} \biggl) dB.
\end{displaymath}
For recovery at each sparsity level, 1000 input $k-$sparse signals $x$ (the nonzero indices are chosen uniformly randomly and the nonzero entries  are drawn with $\sim \mathcal{N}(0, 1)$) are considered.
The recovery is considered to be good if $SNR(x) \geq 100dB$. Simulation 
results on these matrices indicate (shown in Figure \ref{fig:comp_matrix}) that 
the Euler Square matrix gives better performance than the Gaussian and Bernoulli random matrices for 
certain higher sparsity levels. While for other sparsity levels, these matrices give the same recovery performance. It is observed 
that the CS matrices constructed from the Euler Squares of indices $p,s$ for a
prime $p$ and $s= \lfloor{\frac{p}{2}} \rfloor $ to $p-1$, give relatively superior performance than the other choices of $s$ against their Gaussian and Bernoulli counterparts.
\begin{figure}
    \centering
               \includegraphics[width=10cm]{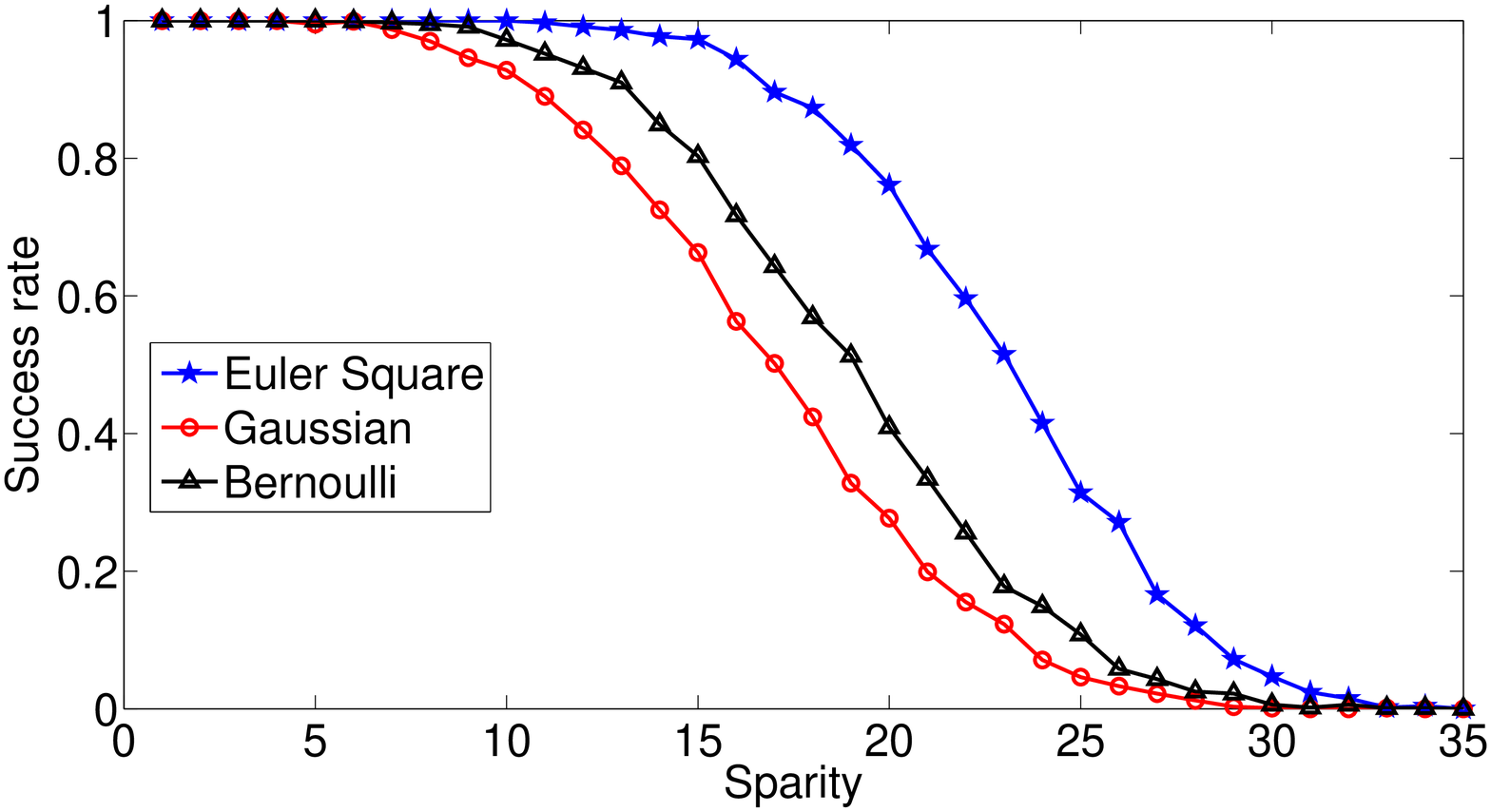}
            
       
         \includegraphics[width=10cm]{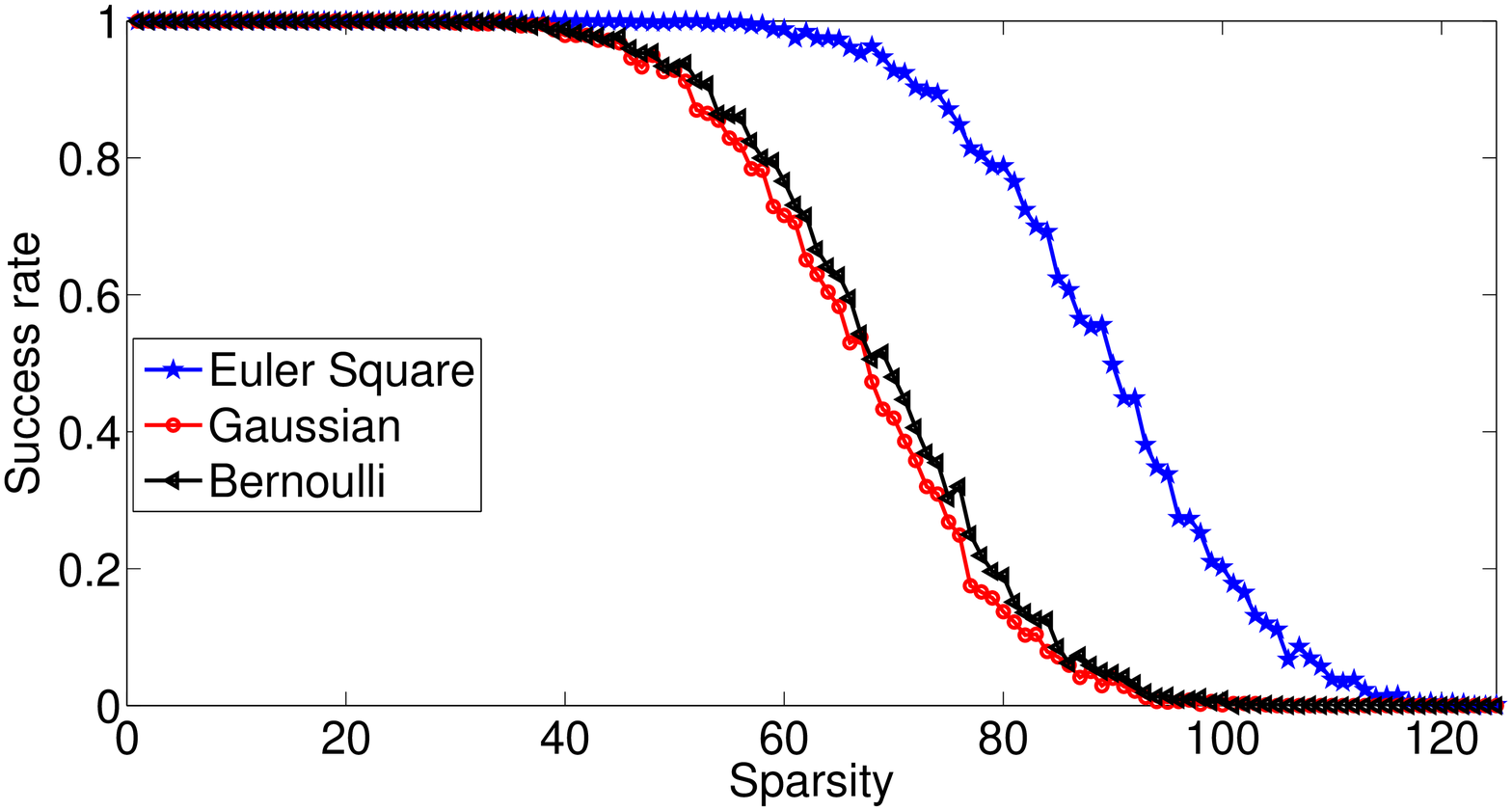}
        \caption{Comparison of the reconstruction performances of Euler Square based, Bernoulli random and the Gaussian random matrices when the matrices are of size (a) $ 55 \times 121$ (top plot) and (b) $230 \times 529$ (bottom plot). These plots indicate that the Euler Square based matrix shows superior performance for some sparsity levels, while for other levels all matrices result in the same performance. The $x$ and $y$ axes in both plots refer respectively to the sparsity level and the success rate (in \% terms). For matrices in (a) and (b), the coherences 0.5226, 0.34 (for Gaussian matrices) 0.52, 0.28 (for Bernoulli matrices) and  0.2, 0.1 (for Euler based matrices). The values of coherence of Euler matrices are equal to their theoretical bound $\mu=\frac{\sqrt{M}}{m}$.}
\label{fig:comp_matrix}
\end{figure}

\subsection{Phase transition}
The phase transition diagrams depict the largest $k$ (with fixed $m$ and $M$) for faithfully recovering $k-$sparse vectors via $l_1$-norm minimization. In Figure~\ref{fig:phase}, the region above the curve is the one in which successful reconstruction is not possible, while the region below consists of points at which successful recovery is possible. Again, 
recovery is considered to be successful if SNR is greater than $100dB$. \
Given a set of points $\delta = \frac{m}{M}$, we have generated phase transition by finding the largest sparsity $k$ such that the successful recovery of 90 percent is achieved by considering the average recovery over 1000 iterations at each point  by using the matrix which gives that particular point. Figure~\ref{fig:phase} provides phase transition for Euler, Gaussian and Bernoulli random matrices for different values of $\delta= \frac{m}{M}$ which were given by the matrices of size $m \times M$, where $m=22, 33, 44, 55, 66, 77, 88, 99, 110$ and M=121. It may be inferred from Figure~\ref{fig:phase} that the Euler Square based recovery is very competitive.

\begin{figure}
\centering
\includegraphics[width=10cm]{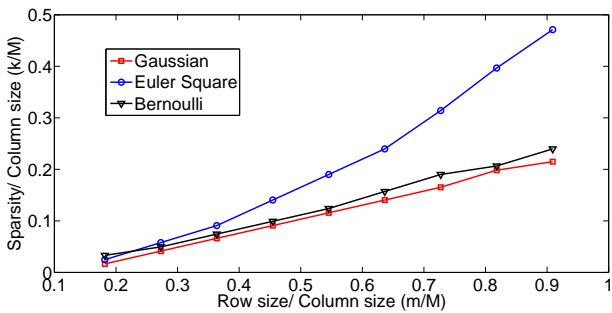}
\caption{Comparison of the reconstruction performances of Euler Square based, the Gaussian random and Bernoulli random matrices  through phase transition. This plot indicates that the Euler Square based matrices provide wider recovery region than their Gaussian and Bernoulli counterparts. The x and y axes in the plot represent $(m/M)$ and $(k/M)$  respectively. This plot is generated for $M=121,\; m=22,33,44,55,66,77,88,99$ and $110.$}
\label{fig:phase}
\end{figure}

\subsection{Reconstruction of images}
The efficacy of Euler Square based matrix is demonstrated using 
image reconstruction from lower dimensional patches, where the patches are 
generated via the sensing matrices (as explained in more detail in the next 
section and Figure~\ref{fig:prop}). Since the CS theory relies on the concept of sparsity, the key property needed of the image is enough sparsity in its original or some transform domain.
\par  As compressed sensing allows for the reconstruction of sparse vector $x$ from its linear measurements $\Phi x$, we demonstrate reconstruction performance via a medical image. The reconstructions shown in Figure~\ref{fig:1} correspond to different down-sampling factors, viz 2.6 and 1.6. The associated reconstruction errors in term of SNR are shown in Table~\ref{fig:Tab1}. Here by down-sampling factor, we mean the ratio of original patch size to reduced patch size which is same as $\frac{M}{m}$ (where, $m \times M$ is  the size of the matrix used for projecting data to the lower dimensional space). From Figure~\ref{fig:1} and Table~\ref{fig:Tab1}, it may be concluded that the Euler Square based matrices provide better reconstruction performance.

\begin{figure}
    \centering
               \includegraphics[height=3.5cm,width=7cm]{{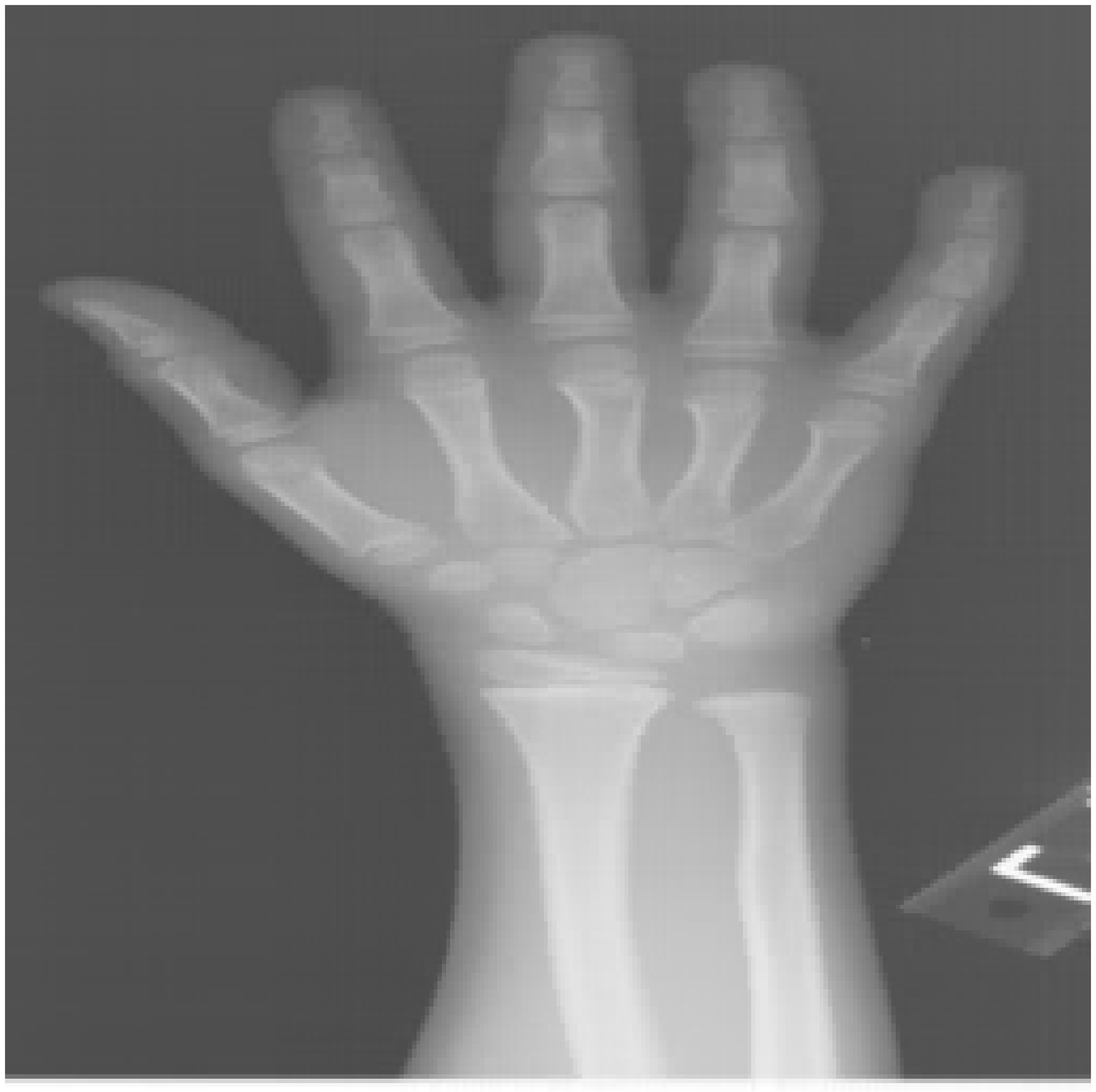}}
							
						\caption{Orginal image}
						\label{fig:a}
            
    \hfill
       
         \includegraphics[height=3.5cm]{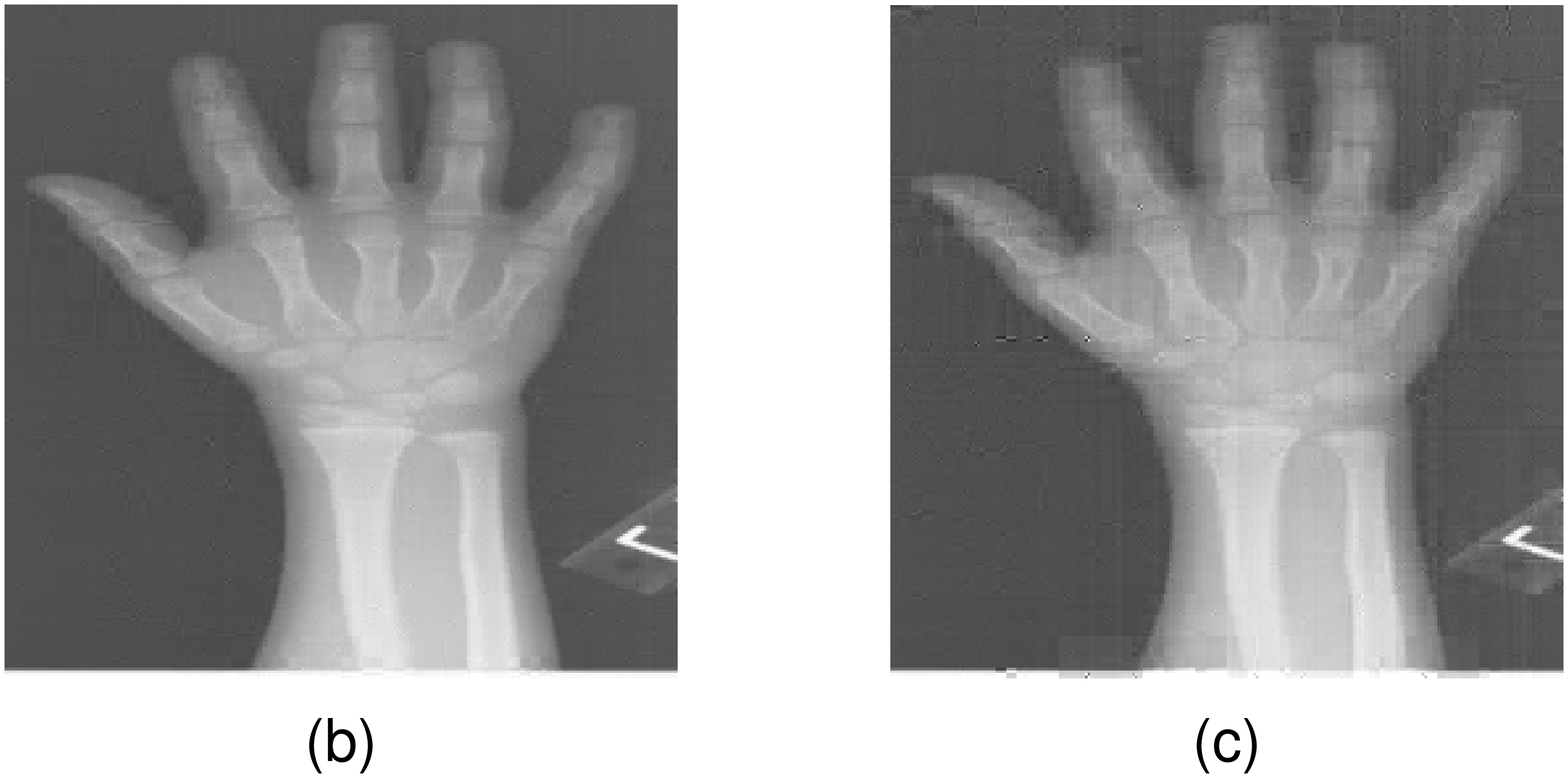}
				\hfill
				\includegraphics[height=3.5cm]{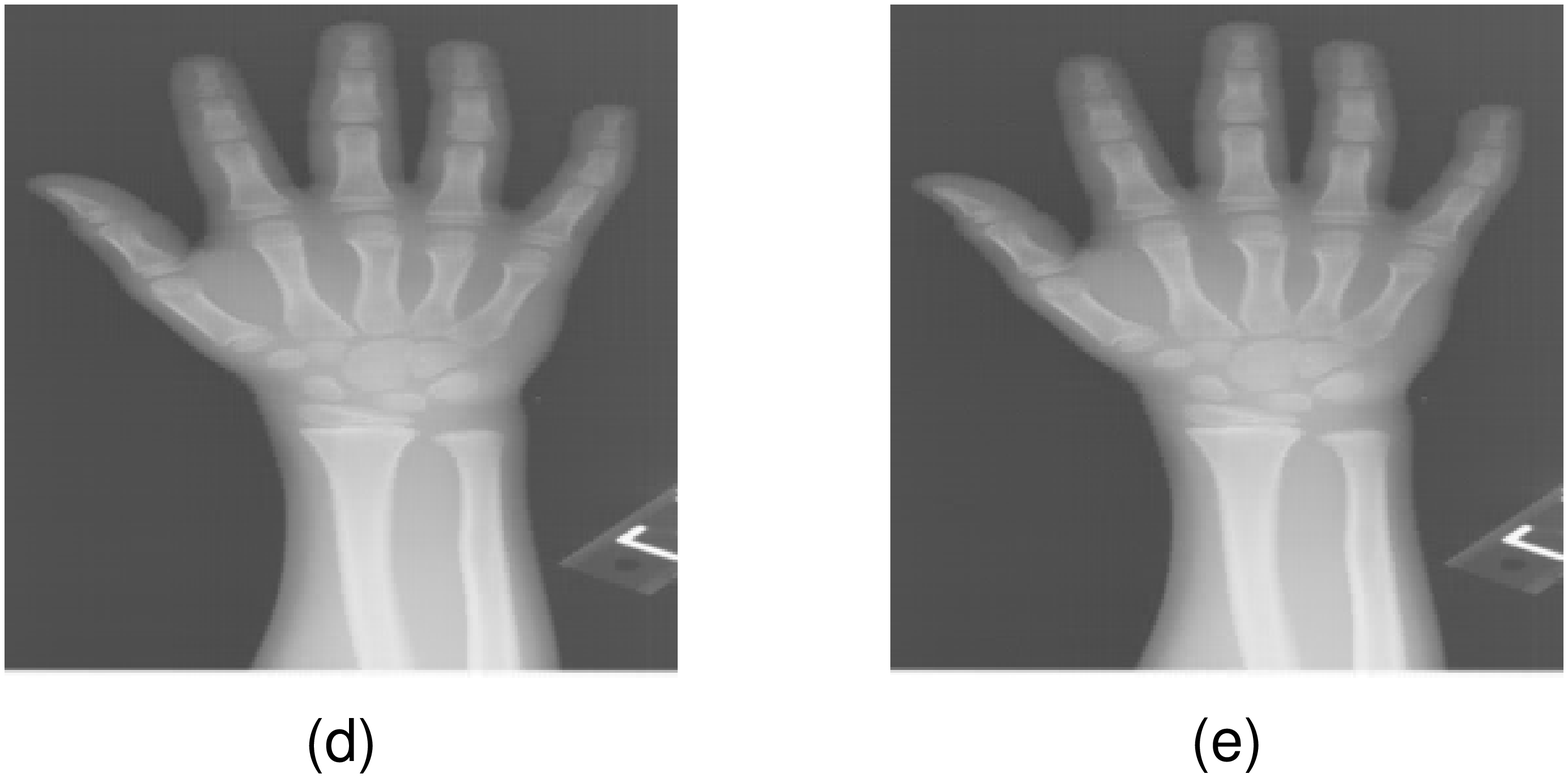}
 \caption{For the original image in Figure \ref{fig:a}, the images in (b) and (d) are those reconstructed via the Euler Square based matrices with down-sampling factors 2.6 and 1.6 respectively. The images in (c) and (e) are those obtained via the corresponding Gaussian matrices. This figure states that Euler Square based CS matrices provide competitive reconstruction performance. The corresponding reconstruction errors are reported in Table  \ref{fig:Tab1}.}
\label{fig:1}
\end{figure} 

\begin{table}[h!t]
\centering
\begin{tabular}{|c|c|c|}
				 \hline
				Down-sampling & Euler recovery & Gaussian recovery \\
				factor ($\frac{M}{m}$) &  SNR error &  SNR error (Average error) \\ \hline
				4&13.36& 13.44 \\
				2.6 & 16.44 & 15.03 \\
				2 & 19.63	 & 18.14 \\
				1.6 & 20.61 & 19.74 \\
			  								\hline
									
 \end{tabular}
    \caption{A comparative error analysis of reconstruction by  Euler based and Gaussian matrices for different down-sampling factors $(\frac{M}{m})$ 4, 2.6, 2, 1.6. The average error over 1000 iterations is reported for Gaussian matrices.}
		\label{fig:Tab1}
	\end{table}			
  
\section{Content based image retrieval (CBIR) using Euler Square matrices} \label{sec:5}
  In this section, the usefulness of binary matrices so constructed for the 
CBIR of medical databases is demonstrated.
  As more and more hospitals use \textit {picture archiving and communication systems} (PACS), the medical imagery world wide is increasingly acquired, transferred and stored digitally \cite{Lehmann2005}.
The increasing dependence on modern medical diagnostic techniques like radiology, histopathology and computerized tomography has led to an explosion in the number of medical images stored in hospitals. Digital image retrieval technique is crucial in the emerging field of medical image databases for clinical decision making process.
It can retrieve images of similar nature (like same modality and disease) and characteristics.  
A typical CBIR system involves 2 steps, namely: 1) feature extraction or dimensionality reduction and 2) retrieval of relevant images through similarity metric. 
\par In fields like tele-medicines \cite{ag_2010, sandy_2013} the dimensionality reduction (DR) is needed for the analysis and classification of medical images, which is followed by the reconstruction at diagnostic center. Though conventional DR approaches \cite{cbir_rev} project data to very low-dimensional spaces than the CS based method, in general, they fail to provide faithful reconstruction from reduced dimension. The DR based on Principal Component Analysis (PCA) has the potential to provide reconstruction as well. Nevertheless, it may be noted that PCA and Compressed Sensing become effective under different sets of conditions \cite{PCA}. Most importantly, the PCA based CBIR is very time-consuming as it is a data-driven approach. In \cite{CS_2013}, a dictionary learning (DL) based CBIR approach is presented, which learns dictionaries in Radon transform domain. The sparsity seeking DL approaches typically exploit the framework of under-determined setting and hence work on some implicit assumptions on database. In CBIR applications, when data bases are not big enough, the sparsity seeking under-determined frame work may not be deployed efficiently. In addition, when labeled data are not used (as is the case with present work), one may not have enough members in a cluster, which prevents the applicability of Dictionary Learning \cite{Elad_2010}. The present CBIR application involving the dimensionality reduction of database members is not prone to these problems.
\par In general the CBIR approaches exploit the properties of database members for feature extraction \cite{CS_2013, Murala_2012}. For example, the CBIR method in \cite{Murala_2012} uses the distribution of edges, and hence appears to work well on databases where the members have pronounced edges. The sparsity based methods, on the other-hand, exploit the presence of inherent sparsity present in most of the databases. In what follows, relevant features are extracted by projecting image content into lower dimensional space. 
  \par The database members $ \{ I_l \}^{N}_{l=1} $ are divided into smaller patches  $ \{ I_{l,p} | \: l = 1,2,..., N \; $and$ \: p = 1,2,...,M' \} $ of equal size. Here, $M'$ stands for the number of patches  being carved out of each of database members. The vectorized versions of  $ I_{l,p} $ are then decomposed into the wavelet domain to generate sparse vectors, say  $ I'_{l,p} $ for each patch. It is to be emphasized here that one may use any transformation that sparsely represents each patch.
   A down-sampled copy of  $ I'_{l,p} $ is generated 
   via the binary sensing matrix $T$ as $ I''_{l,p} = T  I'_{l,p} $.
  If  $ I'_{l,p} $ is sparse enough,  $ I'_{l,p} $ (and consequently  $ I_{l,p} $) can be recovered from the reduced vector  $ I''_{l,p} $ using $l_1$ norm minimization technique. Finally, the feature vector of $l^{th}$ database member is obtained as  $ \{ I''_{l, p} | \: p = 1,2,3,...,M' \} $, which is very small in size compared to  $ \{ I_{l, p} | p = 1,2,...,M' \} $.
  Given the query image $Q$,  its feature vector 
$ \{ Q''_p | p = 1,2,..., M' \}$ is generated similarly. For retrieving the 
images of database that are similar in content with the  query image, 
cross-correlation is employed as a similarity metric (other criteria are also 
possible). The proposed CBIR method is shown in Figure~\ref{fig:prop}.
  
   \begin{figure}[h!t]
                 \centering   
                       \includegraphics[width =\linewidth]{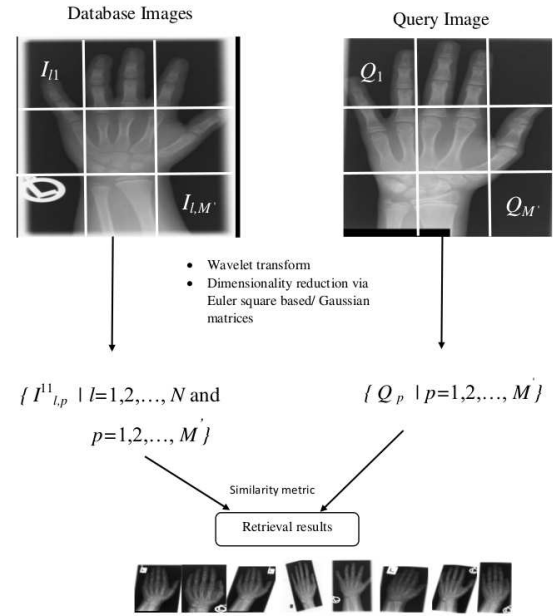}
                     \caption{Block diagram of the proposed CBIR method.}
										\label{fig:prop}
        \end{figure}
  \par The performance of the image retrieval task is measured in terms of 
  recall $\it R$ = $\it N_c$/$\it N_m$ and precision $\it P$ =  $\it N_c$ / $\it (N_c+N_f)$ 
 where $\it N_m$ is the total number of actual (or similar) images, $\it N_c$ is the number of images detected correctly, and $\it N_f $ is the number of false
alarms. A good performance requires both recall and precision to be high, that is, close to unity.
Recall is the portion of total relevant images retrieved whereas precision indicates the capability to retrieve relevant images.

\subsection{Experimental results} \label{sec:6}
 For evaluating proposed CBIR, a database of 821 images present in the form of 12 classes (Figure~\ref{fig:Comp}) has been chosen. These classes of images containing skull, breast, chest and hand etc are taken 	from the popular IRMA database\footnote{www.irma-project.org}. Another 240 are considered as query images.
 
 \par From the database, $ \{ I''_{l,p} | l = 1,2,3,..., 821 ; \: p = 1,2,3,...,64 \} $ and $ \{ I'''_{l,p} | l = 1,2,3,..., 821 ;  p = 1,2,3,...,64 \} $ have been generated.  $I''_{l,p} , I'''_{l,p}$ are obtained using 
 binary and Gaussian matrices respectively for the purpose of comparison of their performances in CBIR.
                
The query image is compared with all database members in compressed domain using cross-correlation as similarity measure for retrieving top 10  similar images.
The precision and recall of binary and Gaussian based feature vector with 10 query images per class are shown in Table \ref{tab:2}. In the simulation work, 
Haar wavelets have been used for sparsely representing each patch.

 \begin{figure}[h!t]
   \centering
                        \includegraphics[width=8cm]{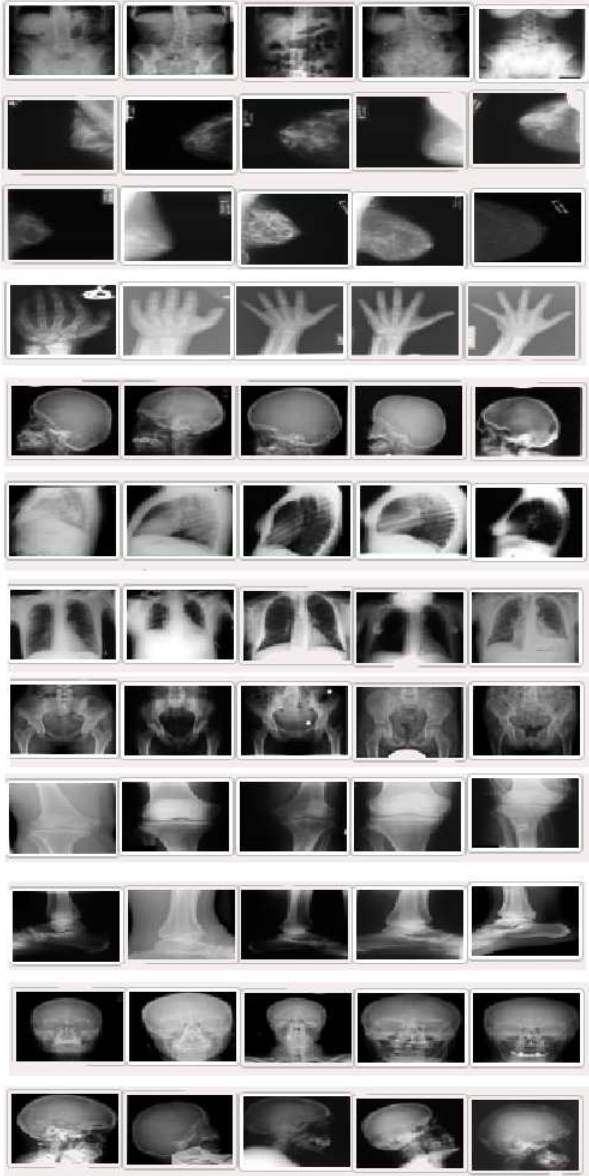}
                    \centering    \caption{Some of the images from each of 12 classes of database.}
										\label{fig:Comp}
                \end{figure}

\begin{table}[h!t]
\centering



\begin{tabular}{|c|c|c|c|c|}
\hline
\multicolumn{1}{|l|}{\multirow{3}{*}{\textbf{Classes}}} & \multicolumn{4}{c|}{\textbf{10 query images per class}}                                    \\ \cline{2-5} 
\multicolumn{1}{|l|}{}                                  & \multicolumn{2}{c|}{\textbf{Euler matrix}} & \multicolumn{2}{c|}{\textbf{Gaussian matrix}} \\ \cline{2-5} 
\multicolumn{1}{|l|}{}                                  & \textbf{Precision}    & \textbf{Recall}    & \textbf{Precision}      & \textbf{Recall}     \\ \hline
\textbf{C1}                                             & 81                    & 47.8               & 90                      & 58.1                \\ \hline
\textbf{C2}                                             & 100                   & 97                 & 100                     & 85.5                \\ \hline
\textbf{C3}                                             & 98                    & 89                 & 97                      & 81                  \\ \hline
\textbf{C4}                                             & 89                    & 58.7               & 92                      & 75.2                \\ \hline
\textbf{C5}                                             & 76                    & 48.9               & 74                      & 45                  \\ \hline
\textbf{C6}                                             & 95                    & 69                 & 98                      & 71                  \\ \hline
\textbf{C7}                                             & 77                    & 39                 & 66                      & 36                  \\ \hline
\textbf{C8}                                             & 91                    & 61                 & 92                      & 45                  \\ \hline
\textbf{C9}                                             & 98                    & 82                 & 95                      & 79                  \\ \hline
\textbf{C10}                                            & 100                   & 69                 & 100                     & 61                  \\ \hline
\textbf{C11}                                            & 100                   & 85                 & 100                     & 82                  \\ \hline
\textbf{C12}                                            & 100                   & 46                 & 94                      & 37                  \\ \hline
\textbf{Average}                                        & \textbf{92}           & \textbf{67.4}      & \textbf{91.5}           & \textbf{63.4}       \\ \hline
\end{tabular}
\caption{ Performances  $($in$ \: \% \: $terms$)$ of the proposed CBIR method based on the  Gaussian and binary  matrices.}
\label{tab:2}
\end{table}

The classification performances of different classes by	 binary and Gaussian matrices are shown in Figures~\ref{fig:conE} and \ref{fig:conG} in terms of a confusion matrix.
The confusion matrix gives the accuracy of the classification results.
The diagonal elements of the confusion matrix indicate if each class is classified correctly. The corresponding nonzero column  entries indicate the presence of mis-classification by the CBIR scheme.
From Table \ref{tab:2}, it may be concluded that the binary based retrieval gives the performance comparable to that of its Gaussian counter part. 
The retrieval performances by both matrices may be improved by fine-tuning the associated parameters (like patch size) and by taking different sparsifying basis. Since the main focus of the present work is to construct CS matrices of general size, we do not go into the details of improvement of performance of CBIR and the reconstruction of database member from the reduced dimensional vectors any further.

 \begin{figure}[h!t]
                 \centering   
                        \includegraphics[width=9.5cm, height=7cm]{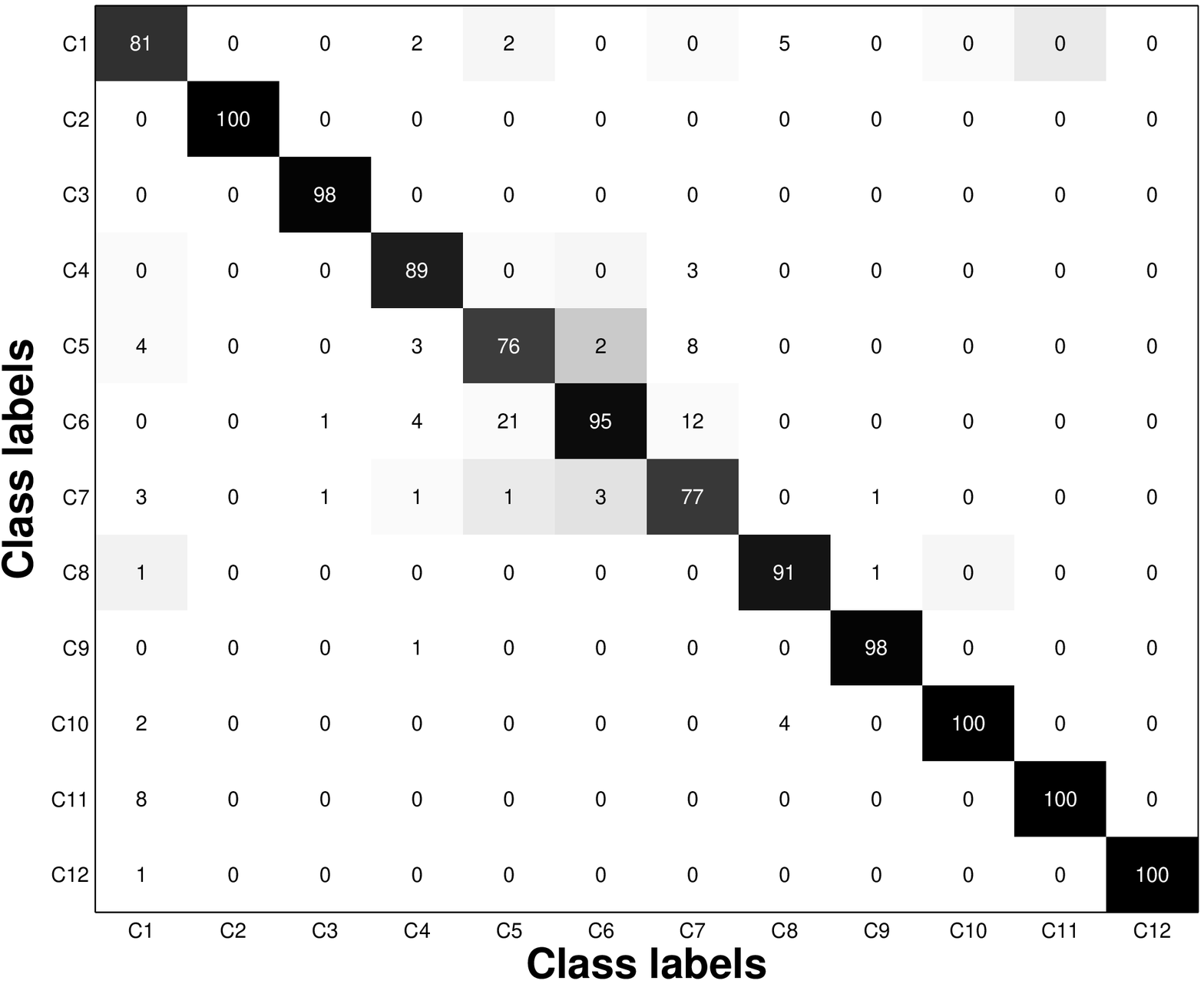}
                    \centering    \caption{Confusion matrix of Euler based binary matrix based CBMIR.}
										\label{fig:conE}
      \end{figure}

 \begin{figure}[h!t]
                 \centering   
                        \includegraphics[width=9.5cm, height=7cm]{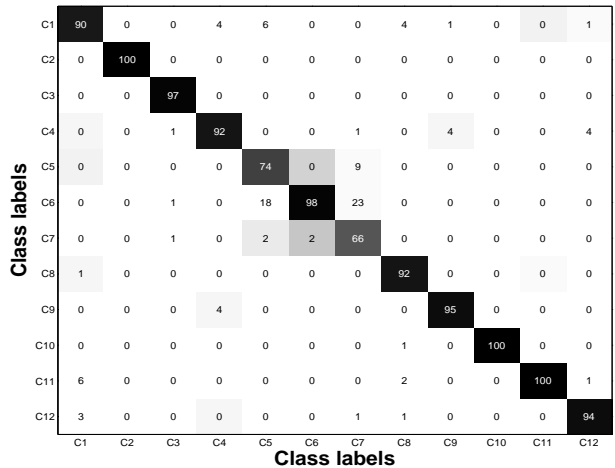}
                    \centering    \caption{Confusion matrix of Gaussian matrix based CBMIR.}
										\label{fig:conG}
      \end{figure}

\section{Concluding Remarks} \label{sec:7}
In this paper, Euler Squares of index $(n,k)$ are used to construct binary sensing matrices of size $nk\times n^2$. Using these and with the help of prime factorization, matrices of general row size with asymptotically optimal column size and good coherence have been constructed. Further, using block wise extension, the column size of the constructed matrices has been enhanced. Simulation results on two test cases show that the generated matrices capture the support of the unknown vector and give better performance than Gaussian matrices. It has also been demonstrated that the proposed binary sensing matrices exploit the pattern of inherent sparsity present in natural images and project data into lower dimensional spaces in such a way that the reduced vectors are useful for the purpose of CBIR.

\section{\bf Acknowledgments} \label{sec:8}
The first author gratefully acknowledges the support (Ref No. 20-6/2009(i)EU-IV) that he receives from UGC, Govt of India. The third author is grateful  to CSIR, Gov. of India (No. 25(219)/13/EMR-II) for the financial support that he received. Authors are thankful to Dr. S. Jana, Dr. M. Srinivas, Mr. Sandip Chandra and Mr. Pradip Sasmal for helpful discussions. 


\end{document}